\documentclass[a4paper,11pt]{article}
\usepackage{geometry}                % See geometry.pdf to learn the layout options. There are lots.
\usepackage[active]{srcltx}
\usepackage{hyperref}

\usepackage{graphicx}
\usepackage{amssymb}
\usepackage{epstopdf}
\usepackage{amsmath}
\usepackage{amsthm}
\usepackage{amssymb}
\usepackage{latexsym}
\usepackage{mathtools}
\usepackage{mathrsfs}
\usepackage{graphics}
\usepackage{latexsym}
\usepackage{psfrag}
\usepackage{import}
\usepackage{verbatim}
\usepackage{enumerate}
\usepackage{enumitem}
\usepackage{graphicx}
\usepackage[usenames]{color}

\newcommand{\R}{\mathbb{R}}

\newcommand{\NN}{\mathcal{N}}
\newcommand{\cH}{\mathcal{H}}

\newcommand{\sgn}{\text{\rm sgn}}

\newcommand{\supp}{\text{\rm supp}}

\newcommand{\diam}{{\rm{diam\,}}}

\newcommand{\ve}{\varepsilon}

\newcommand{\erre}{\mathbb{R}}

\newcommand{\T}{\mathcal{T}}

\renewcommand{\L}{\mathcal{L}}
\newcommand{\F}{\mathcal{F}}
\newcommand{\RCD}{\mathsf{RCD}}

\newcommand{\CD}{\mathsf{CD}}
\newcommand{\PP}{\mathsf{P}}

\newcommand{\Geo}{{\rm Geo}}
\newcommand{\MCP}{\mathsf{MCP}}

\newcommand{\abs}[1]{\left\vert#1\right\vert}

\newcommand{\Real}{\mathbb{R}}

\newcommand{\I}{\mathcal{I}}
\renewcommand{\L}{\mathcal{L}}

\newcommand{\q}{\mathfrak{q}}

\renewcommand{\P}{\mathbb P}

\renewcommand{\P}{\mathcal{P}}

\renewcommand{\H}{\mathcal{H}}

\newcommand{\mm}{\mathfrak m}
\newcommand{\qq}{\mathfrak q}

\newcommand{\ee}{{\rm e}}

\newcommand{\sfd}{\mathsf d}

\newcommand{\Opt}{\mathrm{OptGeo}}

\theoremstyle{plain}
\newtheorem{lemma}{Lemma}[section]
\newtheorem{theorem}[lemma]{Theorem}

\newtheorem{proposition}[lemma]{Proposition}
\newtheorem{corollary}[lemma]{Corollary}
\newtheorem*{theorem*}{Theorem}
\newtheorem*{maintheorem*}{Main Theorem}

\theoremstyle{definition}

\newtheorem{definition}[lemma]{Definition}
\newtheorem*{definition*}{Definition}
\newtheorem{remark}[lemma]{Remark}

\numberwithin{equation}{section}

\title{Isoperimetric inequality \\under Measure-Contraction property}
\author{
Fabio Cavalletti\thanks{SISSA, Trieste 34136, Italy.} \ and 
Flavia Santarcangelo\thanks{SISSA, Trieste 34136, Italy.}}

\date{}                                           % Activate to display a given date or no date

\begin{document}
\maketitle

\begin{abstract}
We prove that if $(X,\sfd,\mm)$ is an essentially non-branching metric measure space with 
$\mm(X)=1$, having Ricci curvature bounded from below by $K$  and dimension bounded above by $N \in (1,\infty)$, understood as a synthetic condition
 called Measure-Contraction property, 
then a sharp isoperimetric inequality  \`a la 
L\'evy-Gromov holds true. Measure theoretic rigidity is also obtained.
\end{abstract}

\bibliographystyle{plain}

%%%%%%%%%%%%%%%%%%%%%%%%%%%%%%%%%%%%%%%%%%%%%%%%%%%%%%%%%%%%%%%%%%%%%%%%%%%%%%
%%%%%%%%%%%%%%%%%%%%%%%%%%%%%%%%%%%%%%%%%%%%%%%%%%%%%%%%%%%%%%%%%%%%%%%%%%%%%%
%%%%%%%%%%%%%%%%%%%%%%%%%%%%%%%%%%%%%%%%%%%%%%%%%%%%%%%%%%%%%%%%%%%%%%%%%%%%%%
%%%%%%%%%%%%%%%%%%%%%%%%%%%%%%%%%%%%%%%%%%%%%%%%%%%%%%%%%%%%%%%%%%%%%%%%%%%%%%

\section{Introduction}

The isoperimetric problem is one of the most classical problems in mathematics; it addresses the following natural problem:
given a space $X$ what is the minimal amount of area needed to enclose a fixed volume $v$. If the space $X$ has a simple structure or has many symmetries the problem can be completely solved and the ``optimal shapes'' can be explicitly described (e.g. Euclidean space and the sphere). In the general case however one cannot hope to obtain a complete solution to the problem and a comparison result is already completely satisfactory. 
Probably the most popular result in this direction is the
L\'evy-Gromov isoperimetric  inequality \cite[Appendix C]{Gro} 
stating that if $E$ is a (sufficiently regular) subset of a Riemannian manifold $(M, g)$ of dimension $N$ and Ricci bounded below by $K > 0$, then
\begin{equation}\label{E:LGiso}
\frac{|\partial E|}{|M|} \geq \frac{|\partial B|}{|S|},
\end{equation}
where $B$ is a spherical cap in the model sphere $S$, i.e. the $N$-dimensional round sphere with constant Ricci curvature equal to $K$, and $|M|,|S|,|\partial E|,|\partial B|$ denote the appropriate $N$ or $N-1$ dimensional volume, and where $B$ is chosen so that $|E|/|M|=|B|/|S|$.

L\'evy-Gromov isoperimetric inequality has been then extended to more general settings; for the scope of this note, the most relevant
progress was the one obtained by E. Milman \cite{Mil} for smooth manifolds with densities, i.e. 
smooth Riemannian manifold whose volume measure has been multiplied by a smooth non-negative integrable density function, having Ricci curvature bounded from below by $K \in \R$ and dimension bounded from above by $N$ in a generalized sense, i.e. verifying the so called Curvature-Dimension condition $\CD(K,N)$ 
introduced in the 1980's by Bakry and \'Emery \cite{BakryStFlour,BakryEmery}.
E. Milman detected a model isoperimetric profile $\I^{\CD}_{K,N,D}$ 
such that if a Riemannian manifold with density verifying $\CD(K,N)$ 
has diameter at most $D > 0$, then the isoperimetric profile function of the weighted manifold is bounded from below by $\I^{\CD}_{K,N,D}$.

After the works of Cordero-Erausquin--McCann--Schmuckenshl\"ager \cite{corderomccann:brescamp}, Otto--Villani \cite{OttoVillaniHWI} and von Renesse--Sturm \cite{VonRenesseSturm}, 
it was realized that the $\CD(K,\infty)$ condition in the smooth setting may be equivalently formulated synthetically as a certain convexity property of an entropy functional along $W_2$ Wasserstein geodesics (associated to $L^2$-Optimal-Transport). This idea led  Lott--Villani \cite{Lott-Villani09} and 
Sturm \cite{Sturm06I, Sturm06II}, to propose a successful (and compatible with the classical one) synthetic definition of $\CD(K,N)$ for a general (complete, separable) metric space $(X,\sfd)$ endowed with a (locally-finite Borel) reference measure $\mm$ (``metric-measure space", or m.m.s.); the theory of m.m.s.'s verifying $\CD(K,N)$ has then extensively developed 
leading to a rich and fruitful approach to the geometry of m.m.s.'s by means of Optimal-Transport \cite{AGS11a, AGS11b, AGS12,Gigli12,EKS2013,AMS2013,MondinoNaber,GigliPasqualetto,KellMondino,BrueSemola}. See also \cite{AmbrosioICM} for a recent account on the topic.
 
Building on the work by Klartag \cite{klartag} and the localization paradigm developed by Payne--Weinberger \cite{PayneWeinberger}, Gromov--Milman \cite{Gromov-Milman} and Kannan--Lov\'asz--Simonovits \cite{KLS}, the first author with Mondino \cite{CM1} managed to extend L\'evy-Gromov-Milman isoperimetric inequality to the class of \emph{essentially non-branching} (see Section \ref{S:backgrounds} for the definition) m.m.s.'s verifying $\CD(K,N)$ with $\mm(X) = 1$; in particular \cite{CM1} proves that 
\begin{equation}\label{E:isoCD}
\mm^{+}(A) \geq \I_{K,N,D}^{\CD}(\mm(A)),
\end{equation}
whenever $A\subset X$ and 
\[
\mm^{+}(A)= \liminf_{\varepsilon \to 0} \frac{\mm(A^{\varepsilon})-\mm(A)}{\varepsilon},
\]
is the Minkowski content of $A$ and $A^{\varepsilon}$ is the $\varepsilon$-enlargement of $A$ given by $A^{\varepsilon}= \{ x \in X \colon \sfd(x,A) < \ve \}$.
The isoperimetric inequality \eqref{E:isoCD} is equivalent to the following inequality
$$
\I_{(X,\sfd,\mm)}(v)\geq \I_{K,N,D}^{\CD}(v),
$$ 
for all $v \in (0,1)$.
Here $\I_{(X,\sfd,\mm)}$ denotes the isoperimetric profile function of the m.m.s. 
$(X,\sfd,\mm)$ defined as follows
$$
\I_{(X,\sfd,\mm)}(v):= \inf\{\, \mm^+(A):\,  A\subset X \text{ Borel},\, \mm(A)=v \}.
$$
In some cases, given a one-dimensional density $h$ defined on the real interval $(a,b)$  integrating to $1$, we will adopt the shorter notation $\I_{h}$ to denote 
the isoperimetric profile function $\I_{\left((a,b), |\cdot|,h\L^{1}\right)}$.

%%%%%%%%%%%%%%%%%%%%%%%%%%%%%%%%%%%%%%%%%%%%%%%%%%%%%%%%%%%%%%%%%%%%%%%%%%%%
%%%%%%%%%%%%%%%%%%%%%%%%%%%%%%%%%%%%%%%%%%%%%%%%%%%%%%%%%%%%%%%%%%%%%%%%%%%%
\subsection{Isoperimetric inequality under Measure-Contraction property}

The Measure Contraction Property $\MCP(K,N)$ was introduced independently by Ohta in \cite{Ohta1} and Sturm in \cite{Sturm06II} as a weaker variant of $\CD(K,N)$. 
Roughly, the idea is to only require the $\CD(K,N)$ condition to hold not for any couple of probability measures $\mu_{0},\mu_{1}$ absolutely continuous with respect to the reference measure $\mm$, but when $\mu_1$ degenerates to a delta-measure at $o \in \supp(\mm)$. 

Still retaining a weaker synthetic lower bound on the Ricci curvature, an upper bound on the dimension and stability in the measured Gromov-Hausdorff sense (see also \cite{OhtaJLMS} for further properties), $\MCP(K,N)$ includes a larger family of spaces than $\CD(K,N)$.
It is now well known for instance that the Heisenberg group equipped with a left-invariant measure, which is the simplest sub-Riemannian structure, does not satisfy any form of $\CD(K,N)$ and do satisfy $\MCP(0,N)$ for a suitable choice of $N$, see \cite{Juillet}. It is worth mentioning that $\MCP$ was first investigated in Carnot groups in \cite{Juillet,Rifford}, see also \cite{BarilariRizzi}.

Recently, interpolation inequalities \`a la 
Cordero-Erausquin--McCann--Schmuckenshl\"ager \cite{corderomccann:brescamp}
have been obtained, under suitable modifications, by Barilari and Rizzi 
\cite{BarilariRizzi2} in the 
ideal sub-Riemannian setting and by Balogh, Krist‡ly and Sipos \cite{Balogh} for the Heisenberg group.
As a consequence, an increasing number of examples of spaces verifying $\MCP$ and 
not $\CD$ is at disposal, e.g.
the Heisenberg group, generalized H-type groups, the Grushin plane and Sasakian structures (for more details, see \cite{BarilariRizzi2}). 
In all the previous examples a sharp isoperimetric inequality 
is not at disposal yet;   
due to lack of regularity of minimizers, sharp isoperimetric inequality has been proved for subclasses of competitors having extra regularity or additional symmetries; in particular, Pansu Conjecture \cite{Pansu} is still unsolved. 
For more details we refer to \cite{Monti,Ritore1,Ritore2,Capogna} and references therein.

\smallskip
In this paper we address the isoperimetric inequality \`a la L\'evy-Gromov
within the class of spaces verifying $\MCP$. In particular, we identify a family of 
one-dimensional $\MCP(K,N)$-densities, each for every choice of $K,N$, volume $v$ and  diameter $D$, not verifying $\CD(K,N)$, and having optimal perimeter; we thus call
the optimal perimeter $\I_{K,N,D}(v)$ and obtain the main result of this note.

\begin{theorem}\label{T:ISOMCP}[Theorem \ref{teo:main}]
Let $K,N\in \R$ with $N>1$ and  
let $(X,\sfd,\mm)$ be an essentially non-branching m.m.s. verifying $\MCP(K,N)$ with 
$\mm(X) = 1$ and having diameter less than $D$. \\
For any $A\subset X$, 
\begin{equation}\label{E:ISOMCP}
\mm^{+}(A) \geq \I_{K,N,D}(\mm(A)).
\end{equation}
Moreover \eqref{E:ISOMCP} is sharp, i.e. for each $v \in [0,1]$, $K,N,D$ there exists 
a m.m.s. $(X,\sfd,\mm)$ with $\mm(X)=1$ and $A\subset X$ 
with $\mm(A) = v$ such that \eqref{E:ISOMCP} is an equality. \\
Finally for each $K,N,D$ and $v \in (0,1)$, $\I_{K,N,D}(v) < \I^{\CD}_{K,N,D}(v)$.
\end{theorem}

Via localization paradigm for $\MCP$-spaces (see Section \ref{S:final} for details), following \cite{klartag,CM1}, the proof of Theorem \ref{T:ISOMCP}
is reduced to the proof of the corresponding statement in the one-dimensional setting.
However, contrary to the $\CD$ framework, due to lack of any form of concavity, 
the isoperimetric problem 
for a general one-dimensional $\MCP(K,N)$-density seems to be out of reach.
We instead directly exhibit, for each $K,N,D$ and $v$, an optimal one-dimensional $\MCP(K,N)$-density, 
denoted by $h_{K,N,D,v}$ that will be optimal only for that choice of $K,N,D$ and $v$.
In particular, 
\begin{equation}\label{E:ISOexplicit}
\I_{K,N,D}(v) = 
\begin{cases}
h_{K,N,D,v}(a_{K,N,D}(v)), & K \leq 0,  \\
{\displaystyle \min_{D'\leq D}h_{K,N,D',v}(a_{K,N,D'}(v))}, & K>0. 
\end{cases}
\end{equation}
where $a_{K,N,D}(v)$ is the unique point of $[0,D]$ such that 
$\int_{[0,a_{K,N,D}(v)]}h_{K,N,D,v}(x)\,dx = v$; in particular 
$$
\I_{h_{K,N,D,v}}(v) = h_{K,N,D,v}(a_{K,N,D}(v)),
$$
for all $K,N,D$ and $v$.
To explain \eqref{E:ISOexplicit}, we underline that for each $K,N,D$ and $v$, 
$h_{K,N,D,v}(a_{K,N,D}(v))$ is the optimal perimeter when minimization is 
constrained to all one-dimensional $\MCP(K,N)$-densities (integrating to 1) 
having support of \emph{exactly}  length $D$, see Theorem \ref{T:comparison}. 
Denoting the optimal value of the latter minimization problem by 
$\widetilde{\I}_{K,N,D}(v)$, the previous sentence reads as 
\begin{equation}\label{E:onediso}
\widetilde{\I}_{K,N,D}(v)= \I_{h_{K,N,D,v}}(v). 
\end{equation}
Hence 
\eqref{E:ISOexplicit} is a direct consequence of the following fact: 
$\widetilde{\I}_{K,N,D}(v)$ is strictly decreasing as a function of $D$ only if $K\leq0$, showing a remarkable difference with the $\CD$-framework (see \cite{Mil}).

\smallskip
The rigidity property of L\'evy-Gromov isoperimetric inequality is a well-known fact: if a Riemannian manifold verifies the equality case in \eqref{E:LGiso} then it is isometric to the round sphere of the correct dimension \cite{Gro}; if equality is attained in \eqref{E:isoCD} and the metric measure space verifies the stronger $\RCD(K,N)$ condition (see \cite{AGS11a, AGS11b,Gigli12,EKS2013,AMS2013} and references therein), then it is isomorphic in the metric-measure sense to a spherical suspension (see \cite{CM1} for details). 
At the present generality, i.e. the class of m.m.s.'s verifying $\MCP(K,N)$, 
competitors are less regular and a weaker rigidity is valid. 

In particular, the proof of Theorem \ref{T:comparison} is sufficiently stable to imply one-dimensional rigidity (Theorem \ref{T:rigidityonedim}), 
valid for each choice of $K,N,D$ and $v$. 
Building on this and on the monotonicity in $D$ of $\widetilde{\I}_{K,N,D}(v)$,
we show that whenever $K \leq 0$ the optimal metric measure space has a product structurein a measure theoretic sense (see Theorem \ref{T:rigidityfull} for the precise result).

\smallskip
We conclude the Introduction presenting the structure of the paper.
Section \ref{S:backgrounds} contains some basics on the theory of m.m.s.'s 
verifying synthetic lower bounds on Ricci curvature. Section \ref{S:oned}
proves the new main one-dimensional facts on $\MCP(K,N)$-densities; 
Section \ref{S:final} contains the main results of the paper and a general 
overview on localization technique.

%%%%%%%%%%%%%%%%%%%%%%%%%%%%%%%%%%%%%%%%%%%%%%%%%%%%%%%%%%%%%%%%%%%%%%%%%%%%%%
%%%%%%%%%%%%%%%%%%%%%%%%%%%%%%%%%%%%%%%%%%%%%%%%%%%%%%%%%%%%%%%%%%%%%%%%%%%%%%
%%%%%%%%%%%%%%%%%%%%%%%%%%%%%%%%%%%%%%%%%%%%%%%%%%%%%%%%%%%%%%%%%%%%%%%%%%%%%%
%%%%%%%%%%%%%%%%%%%%%%%%%%%%%%%%%%%%%%%%%%%%%%%%%%%%%%%%%%%%%%%%%%%%%%%%%%%%%%

\section{Backgrounds}\label{S:backgrounds}

A triple $(X,\sfd,\mm)$  is called  metric measure space (or m.m.s.)  if $(X,\sfd)$ a Polish space (i.e. a  complete and separable metric space) and $\mm$ is  a positive Radon measure over $X$; 
in this work however we will always assume $\mm(X)=1$.  

We denote by 
$$
\Geo(X) : = \{ \gamma \in C([0,1], X):  \sfd(\gamma_{s},\gamma_{t}) = |s-t| \sfd(\gamma_{0},\gamma_{1}), \text{ for every } s,t \in [0,1] \}.
$$
the space of constant speed geodesics. The metric space $(X,\sfd)$ is a \emph{geodesic space} if and only if for each $x,y \in X$ 
there exists $\gamma \in \Geo(X)$ so that $\gamma_{0} =x, \gamma_{1} = y$.
For complete geodesic spaces, local compactness is equivalent to properness (a metric space is proper if every closed ball is compact).

\smallskip

$\mathcal{P}(X)$ denotes the  space of all Borel probability measures over $X$ and with  $\mathcal{P}_{2}(X)$ the space of probability measures with finite second moment.
$\mathcal{P}_{2}(X)$ can be  endowed with the $L^{2}$-Kantorovich-Wasserstein distance  $W_{2}$ defined as follows:  for $\mu_0,\mu_1 \in \mathcal{P}_{2}(X)$,  set
\begin{equation}\label{eq:W2def}
  W_2^2(\mu_0,\mu_1) := \inf_{ \pi} \int_{X\times X} \sfd^2(x,y) \, \pi(dxdy),
\end{equation}
where the infimum is taken over all $\pi \in \mathcal{P}(X \times X)$ with $\mu_0$ and $\mu_1$ as the first and the second marginal.
The space $(X,\sfd)$ is geodesic  if and only if the space $(\mathcal{P}_2(X), W_2)$ is geodesic. 

\medskip

 For any $t\in [0,1]$,  let ${\rm e}_{t}$ denote the evaluation map: 
$$
  {\rm e}_{t} : \Geo(X) \to X, \qquad {\rm e}_{t}(\gamma) : = \gamma_{t}.
$$
Any geodesic $(\mu_t)_{t \in [0,1]}$ in $(\mathcal{P}_2(X), W_2)$  can be lifted to a measure $\nu \in {\mathcal {P}}(\Geo(X))$, 
so that $({\rm e}_t)_\sharp \, \nu = \mu_t$ for all $t \in [0,1]$. 
\\Given $\mu_{0},\mu_{1} \in \mathcal{P}_{2}(X)$, we denote by 
$\Opt(\mu_{0},\mu_{1})$ the space of all $\nu \in \mathcal{P}(\Geo(X))$ for which $({\rm e}_0,{\rm e}_1)_\sharp\, \nu$ 
realizes the minimum in \eqref{eq:W2def}. Such a $\nu$ will be called \emph{dynamical optimal plan}. If $(X,\sfd)$ is geodesic, then the set  $\Opt(\mu_{0},\mu_{1})$ is non-empty for any $\mu_0,\mu_1\in \mathcal{P}_2(X)$.
\\We will also consider the subspace $\mathcal{P}_{2}(X,\sfd,\mm)\subset \mathcal{P}_{2}(X)$
formed by all those measures absolutely continuous with respect with $\mm$.
\medskip

In the paper we will only consider essentially non-branching spaces, let us recall their definition (introduced in \cite{RS2014}). 

A set $G \subset \Geo(X)$ is a set of non-branching geodesics if and only if for any $\gamma^{1},\gamma^{2} \in G$, it holds:
$$
\exists \;  \bar t\in (0,1) \text{ such that } \ \forall t \in [0, \bar t\,] \quad  \gamma_{ t}^{1} = \gamma_{t}^{2}   
\quad 
\Longrightarrow 
\quad 
\gamma^{1}_{s} = \gamma^{2}_{s}, \quad \forall s \in [0,1].
$$
\begin{definition}\label{def:ENB}
A metric measure space $(X,\sfd, \mm)$ is \emph{essentially non-branching} (e.n.b. for short) if and only if for any $\mu_{0},\mu_{1} \in \mathcal{P}_{2}(X)$,
with $\mu_{0},\mu_{1}$ absolutely continuous with respect to $\mm$, any element of $\Opt(\mu_{0},\mu_{1})$ is concentrated on a set of non-branching geodesics.
\end{definition}
It is clear that if $(X,\sfd)$ is a  smooth Riemannian manifold  then any subset $G \subset \Geo(X)$  is a set of non-branching geodesics, in particular any smooth Riemannian manifold is essentially non-branching.

It is worth stressing that 
the restriction to essentially non-branching m.m.s.'s is done to avoid pathological cases:
as an example of possible pathological behaviour we mention the failure of the local-to-global property of $\CD(K,N)$ within this class of spaces; in particular, a heavily-branching m.m.s. verifying $\CD_{loc}(0,4)$ which does not verify $\CD(K,N)$ for any fixed $K \in \R$ and $N \in [1,\infty]$ was constructed by Rajala in \cite{R2016}, while the local-to-global property of $\CD(K,N)$
has been recently proved to hold \cite{CMi16} for essentially non-branching m.m.s.'s.

%%%%%%%%%%%%%%%%%%%%%%%%%%%%%%%%%%%%%%%%%%%%%%%%%%%%%%%%%%%%%%%%%%%%%%
%%%%%%%%%%%%%%%%%%%%%%%%%%%%%%%%%%%%%%%%%%%%%%%%%%%%%%%%%%%%%%%%%%%%%%
%%%%%%%%%%%%%%%%%%%%%%%%%%%%%%%%%%%%%%%%%%%%%%%%%%%%%%%%%%%%%%%%%%%%%%

\subsection{Measure-Contraction Property}

We briefly describe the $\MCP$ condition encapsulating generalized 
Ricci curvature lower bounds coupled 
with generalized dimension upper bounds.

\begin{definition}[$\sigma_{K,\NN}$-coefficients] \label{def:sigma}
Given $K \in \Real$ and $\NN \in (0,\infty]$, define:
\[
D_{K,\NN} := \begin{cases}  \frac{\pi}{\sqrt{K/\NN}}  & K > 0 \;,\; \NN < \infty \\ +\infty & \text{otherwise} \end{cases} .
\]
In addition, given $t \in [0,1]$ and $0 < \theta < D_{K,\NN}$, define:
\[
\sigma^{(t)}_{K,\NN}(\theta) :=  
\begin{cases}   
\frac{\sin(t \theta \sqrt{\frac{K}{\NN}})}{\sin(\theta \sqrt{\frac{K}{\NN}})}  & K > 0 \;,\; \NN < \infty \\
t & K = 0 \text{ or } \NN = \infty \\
 \frac{\sinh(t \theta \sqrt{\frac{-K}{\NN}})}{\sinh(\theta \sqrt{\frac{-K}{\NN}})} & K < 0 \;,\; \NN < \infty 
\end{cases} ,
\]
and set $\sigma^{(t)}_{K,\NN}(0) = t$ and $\sigma^{(t)}_{K,\NN}(\theta) = +\infty$ for $\theta \geq D_{K,\NN}$. 
Finally given $K \in \Real$ and $N \in (1,\infty]$, define:
\[
\tau_{K,N}^{(t)}(\theta) := t^{\frac{1}{N}} \sigma_{K,N-1}^{(t)}(\theta)^{1 - \frac{1}{N}} .
\]
When $N=1$, set $\tau^{(t)}_{K,1}(\theta) = t$ if $K \leq 0$ and $\tau^{(t)}_{K,1}(\theta) = +\infty$ if $K > 0$. 
\end{definition}

\begin{definition}[$\MCP(K,N)$] \label{D:Ohta1}
A m.m.s. $(X,\sfd,\mm)$ is said to satisfy $\MCP(K,N)$ if for any $o \in \supp(\mm)$ and  $\mu_0 \in \P_2(X,\sfd,\mm)$ of the form $\mu_0 = \frac{1}{\mm(A)} \mm\llcorner_{A}$ for some Borel set $A \subset X$ with $0 < \mm(A) < \infty$ (and with $A \subset B(o, \pi \sqrt{(N-1)/K})$ if $K>0$), there exists $\nu \in \Opt(\mu_0, \delta_{o} )$ such that:
\begin{equation} \label{eq:MCP-def}
\frac{1}{\mm(A)} \mm \geq (\ee_{t})_{\sharp} \big( \tau_{K,N}^{(1-t)}(\sfd(\gamma_{0},\gamma_{1}))^{N} \nu(d \gamma) \big) \;\;\; \forall t \in [0,1] .
\end{equation}
\end{definition}

If $(X,\sfd,\mm)$ is a m.m.s. verifying $\MCP(K,N)$, then $(\supp(\mm),\sfd)$  is Polish, proper and it is a geodesic space. 
With no loss in generality for our purposes we will assume that $X = \supp(\mm)$. 
Many additional results on the structure of $W_{2}$-geodesics can be obtained just from the $\MCP$ condition together with the essentially non-branching assumption (see \cite{CM16}).

To conclude, referring to \cite{Ohta1, Sturm06II} for more general results, we report the following important fact
\cite[Theorem 3.2]{Ohta1}: if $(M,g)$ is $n$-dimensional Riemannian manifold 
with $n\geq 2$, the m.m.s. $(M,d_{g},vol_{g})$ verifies $\MCP(K,n)$ if and only if $Ric_{g} \geq K g$, where $d_{g}$ is the geodesic distance induced by $g$ and $v_{g}$ the volume measure.

%%%%%%%%%%%%%%%%%%%%%%%%%%%%%%%%%%%%%%%%%%%%%%%%%%%%%%%%%%%%%%%%
%%%%%%%%%%%%%%%%%%%%%%%%%%%%%%%%%%%%%%%%%%%%%%%%%%%%%%%%%%%%%%%%

\medskip
A relevant case for our purposes (due to the crucial use of the localization technique)  is the one of one-dimensional spaces $(X,\sfd,\mm)=(I,|\cdot|, h \L^1)$.
It is a standard fact that the m.m.s. $(I,|\cdot|,h\L^{1})$ verifies $\MCP(K,N)$ if and only if the non-negative Borel function $h$ satisfies the following inequality: 
\begin{equation}\label{E:MCPdef}
 h(t x_1 + (1-t) x_0) \geq \sigma^{(1-t)}_{K,N-1}(\abs{x_1-x_0})^{N-1} h(x_0).
\end{equation}
for all $x_0,x_1 \in I$ and $t \in [0,1]$. We will call $h$ an 
$\MCP(K,N)$-density.

\smallskip
Inequality \eqref{E:MCPdef} implies several known properties that we recall for readers convenience.
To write them in a unified way, we 
define for $\kappa \in \erre$ the function $s_{\kappa} : [0,+\infty) \to \erre$ (on $[0,\pi/\sqrt{\kappa})$ if $\kappa>0$) 
\begin{equation}\label{E:sk}
s_{\kappa}(\theta):= 
\begin{cases}
(1/\sqrt{\kappa})\sin(\sqrt{\kappa}\theta) & {\rm if}\ \kappa>0, \crcr
\theta & {\rm if}\ \kappa=0, \crcr
(1/\sqrt{-\kappa})\sinh(\sqrt{-\kappa}\theta) &{\rm if}\ \kappa<0.
\end{cases}
\end{equation}
For the moment we confine ourselves to the case $I = (a,b)$ with $a,b \in \R$; hence
\eqref{E:MCPdef} implies (actually  is equivalent to)
\begin{equation}\label{E:MCPdef2}
\left( \frac{s_{K/(N-1)}(b - x_{1}  )}{s_{K/(N-1)}(b - x_{0}  )} \right)^{N-1} 
\leq \frac{h(x_{1} ) }{h (x_{0})} \leq 
\left( \frac{s_{K/(N-1)}( x_{1} -a  )}{s_{K/(N-1)}( x_{0} -a  )} \right)^{N-1}, 
\end{equation}
for $x_{0} \leq x_{1}$. 
In particular, $h$ is locally Lipschitz in the interior of $I$ and continuous up to the boundary.
The next lemma was stated and proved in  \cite[Lemma A.8]{CMi16}  under the $\CD$ condition; as the proof only uses $\MCP(K,N)$ we report it in this more general version.

\begin{lemma} \label{lem:apriori0}
Let $h$ denote a $\MCP(K,N)$ density on a finite interval $(a,b)$, $N \in (1,\infty)$, which integrates to $1$. Then:
\begin{equation}\label{eq:suphdiam}
\sup_{x \in (a,b)} h(x) \leq \frac{1}{b-a} \begin{cases} N & K \geq 0  \\ (\int_0^1 (\sigma^{(t)}_{K,N-1}(b-a))^{N-1} dt)^{-1} & K < 0 \end{cases} .
\end{equation}
In particular, for fixed $K$ and $N$, $h$ is uniformly bounded from above as long as $b-a$ is uniformly bounded away from $0$ (and from above if $K < 0$).
\end{lemma}

\bigskip

%%%%%%%%%%%%%%%%%%%%%%%%%%%%%%%%%%%%%%%%%%%%%%%%%%%%%%%%%%%%%%%%%%%%%%%%%%%%%%
%%%%%%%%%%%%%%%%%%%%%%%%%%%%%%%%%%%%%%%%%%%%%%%%%%%%%%%%%%%%%%%%%%%%%%%%%%%%%%
%%%%%%%%%%%%%%%%%%%%%%%%%%%%%%%%%%%%%%%%%%%%%%%%%%%%%%%%%%%%%%%%%%%%%%%%%%%%%%
%%%%%%%%%%%%%%%%%%%%%%%%%%%%%%%%%%%%%%%%%%%%%%%%%%%%%%%%%%%%%%%%%%%%%%%%%%%%%%

\section{One-dimensional analysis}\label{S:oned}

The isoperimetric problem for a one-dimensional density $h$ verifying $\MCP(K,N)$ for some $K,N\in \R$ and $N > 1$ will be addressed in this section.

Without loss of generality we can assume $h$ to be defined over $[0,D]$ (recall that $D\leq \pi \sqrt{(N-1)/K}$, whenever $K >0$). 
Recall that the case $K > 0$ and $D=\pi \sqrt{(N-1)/K}$ is trivial as 
\eqref{E:MCPdef2} forces the density to coincide with the model 
density $\sin^{N-1}(t)$ (that in particular is also a $\CD(K,N)$-density).

\begin{proposition}[Lower Bound]\label{P:lower}
Define the following strictly positive function
$$
 f_{K,N,D}(x) : = \left( 
\int_{(0,x)}\left( \frac{s_{K/(N-1)}(D -  y)}{s_{K/(N-1)}(D - x  )} \right)^{N-1}  dy 
+\int_{(x,D)} \left( \frac{s_{K/(N-1)}(y)}{s_{K/(N-1)}(x  )} \right)^{N-1} dy \right)^{-1}
$$
for $x \in (0,D)$ and equal $0$ for $x = 0,D$. Then 
\begin{itemize}
\item[i)] $f_{K,N,D}$ is strictly increasing over $(0,D/2)$;
\item[ii)] $f_{K,N,D}(x) = f_{K,N,D}(D- x)$;
\item[iii)] if $h : [0,D] \to \R$ is an $\MCP(K,N)$-density integrating to $1$, then $h(x) \geq f_{K,N,D}(x)$.
\end{itemize}
\end{proposition}

\begin{proof}
The second claim is straightforward to check. 
For the first one, being $f_{K,N,D}$ a smooth function and strictly positive in $(0,D)$, it will be enough to show that $f'_{K,N,D} (x) = 0$ has no solution for 
$x \in (0,D/2)$. Imposing $f'_{K,N,D} (x) = 0$ is equivalent to 
$$
\frac{s_{K/(N-1)}'(D-x)}{s_{K/(N-1)}^{N}(D-x)} \int_{(0,x)} s_{K/(N-1)}^{N-1}(D-y)\, dy
=
\frac{s_{K/(N-1)}'(x)}{s_{K/(N-1)}^{N}(x)} \int_{(x,D)} s_{K/(N-1)}^{N-1}(y) \, dy,
$$
that can be rewritten as 
$$
\frac{s_{K/(N-1)}'(D-x)}{s_{K/(N-1)}^{N}(D-x)} \int_{(D-x,D)} s_{K/(N-1)}^{N-1}(y)\, dy
=
\frac{s_{K/(N-1)}'(x)}{s_{K/(N-1)}^{N}(x)} \int_{(x,D)} s_{K/(N-1)}^{N-1}(y) \, dy.
$$
Since $D-x \geq x$, the previous identity implies 
\begin{equation}\label{E:aux1}
\frac{|s_{K/(N-1)}'(D-x)|}{s_{K/(N-1)}^{N}(D-x)} > \frac{|s_{K/(N-1)}'(x)|}{s_{K/(N-1)}^{N}(x)}. 
\end{equation}
For $K = 0$, \eqref{E:aux1} becomes 
$$
\frac{1}{(D-x)^{N}} > \frac{1}{x^{N}}, 
$$
giving a contradiction. For negative $K = - (N-1)$ (the other negative cases follow similarly) \eqref{E:aux1} implies
$$
\frac{\cosh(D-x)}{\sinh(D-x)^{N}} > \frac{\cosh(x)}{\sinh(x)^{N}}
$$
forcing 
$$
\frac{\cosh(D-x)}{\sinh(D-x)} > \left(\frac{\sinh(D-x)}{\sinh(x)}\right)^{N-1} \frac{\cosh(x)}{\sinh(x)} >\frac{\cosh(x)}{\sinh(x)}, 
$$
giving a contradiction with monotonicity of $\tanh$. 
Finally, for $K = N-1$, \eqref{E:aux1} becomes
$$
\frac{\cos(D-x)}{\sin(D-x)^{N}} > \frac{\cos(x)}{\sin(x)^{N}}, 
\qquad \sgn(\cos(D-x)) = \sgn (\cos(x));
$$
the second identity implies that $x < D-x < \pi/2$ or $\pi/2 < x < D - x$.
The second case would imply that $D > 2x > \pi$ giving a contradiction. Hence we are left with $x < D-x < \pi/2$: 
$$
1 > \frac{\cos(D-x)}{\cos x} > \left(\frac{\sin(D-x)}{\sin x}\right)^{N}, 
$$
giving a contradiction.

The third claim follows simply observing that \eqref{E:MCPdef2} gives 
\begin{align*}
1 = &~ \int_{(0,x)} h(y) \, dy +\int_{(x,D)} h(y) \, dy   \crcr
\leq &~ \frac{h(x)}{s_{K/(N-1)}^{N-1}(D-x)}\int_{(0,x)} s_{K/(N-1)}^{N-1}(D-y) \, dy +\frac{h(x)}{s_{K/(N-1)}^{N-1}(x)}\int_{(x,D)} s_{K/(N-1)}^{N-1}(y) \, dy,   
\end{align*}
and the claim is proved.
\end{proof}

Starting from the lower bound of Proposition \ref{P:lower},
we define a distinguished family of $\MCP(K,N)$ densities, depending on four parameters, that will be the model one-dimensional isoperimetric density: 
\begin{equation}
\label{equ:ha}
h_{K,N,D}^{a}(x) : = 
f_{K,N,D}(a)
\begin{cases}
{\displaystyle \left(\frac{s_{K/(N-1)}(D-x)}{s_{K/(N-1)}(D-a)}\right)^{N-1}}, 
		& x \leq a, \crcr \\
{\displaystyle \left(\frac{s_{K/(N-1)}(x)}{s_{K/(N-1)}(a)}\right)^{N-1}}, 
		& x \geq a.
\end{cases}
\end{equation}
Notice that $h_{K,N,D}^{D-a}(D-x) = h_{K,N,D}^{a}(x)$ and 
$$
h_{K,N,D}^{a} (z D/D') = h_{(D/D')^{2}K,N,D'}^{aD'/D}(z),
$$ 
showing that it will no be restrictive to assume for some of the next proofs 
$K = N-1$ or $K = -(N-1)$, letting $D$ vary.

\begin{corollary}[Rigidity of lower bound]\label{C:lowerboundrigid}
Let $h : [0,D] \to \R$ be a $\MCP(K,N)$-density integrating to $1$. 
Assume $h(y) = f_{K,N,D}(y)$ for some $y \in (0,D)$; then 
$h = h^{y}_{K,N,D}$.
\end{corollary}
\begin{proof}
From the proof Proposition \ref{P:lower}, point $iii)$, and \eqref{E:MCPdef2}
one deduces that 
$$
h(x) = 
h(y)
\begin{cases}
{\displaystyle \left(\frac{s_{K/(N-1)}(D-x)}{s_{K/(N-1)}(D-a)}\right)^{N-1}}, 
		& x \leq a, \crcr \\
{\displaystyle \left(\frac{s_{K/(N-1)}(x)}{s_{K/(N-1)}(a)}\right)^{N-1}}, 
		& x \geq a.
\end{cases}
$$
The claim then follows.
\end{proof}

To avoid cumbersome notation, the dependence of $h_{K,N,D}^{a}$ on $K,N,D$ will be omitted and we will use $h_{a}$.

\begin{lemma}\label{L:specialdensity}
For every $a \in (0,D)$, the function $h_{a}$ integrates to $1$ and it is an $\MCP(K,N)$-density.
\end{lemma}

\begin{proof}
Each $h_{a}$ has by definition integral $1$. 
To check $\MCP(K,N)$ it will be enough to verify that the inequality \eqref{E:MCPdef2} is satisfied.  

\smallskip
We start observing that the function
\begin{equation}
\label{E:decreasing}
\frac{s_{K/(N-1)}(D-\cdot\,)}{s_{K/(N-1)}(\cdot)}
\end{equation}
is  decreasing   in ${[0,D]}$; this will be proved showing its first derivative to be negative:
$$
 \frac{s_{K/(N-1)}'(D-a)}{s_{K/(N-1)}(a)}  + \frac{s_{K/(N-1)}(D-a)s_{K/(N-1)}'(a)}{s^{2}_{K/(N-1)}(a)} \geq 0.
$$
The previous inequality is straightforward for $K \leq 0$; for $K> 0$, assuming without loss of generality   $K = N-1$, it  reduces to
$\sin(a) \cos(D-a) + \sin(D-a) \cos(a) = \sin(D)\geq0$, that is always verified with the strict inequality except for the trivial  case  $D=\pi$ (where the function \eqref{E:decreasing} is identically equal to one).\\
 Using the result just obtained, we are able to check \eqref{E:MCPdef2} distinguishing three cases.

\noindent
If $x_{0}\leq x_{1} \leq a$: 
$$
 \left(\frac{s_{K/(N-1)}(D-x_{1})}{s_{K/(N-1)}(D-x_{0})}\right)^{N-1}
= \frac{h_{a}(x_{1})}{h_{a}(x_{0})} 
\leq \left(\frac{s_{K/(N-1)}(x_{1})}{s_{K/(N-1)}(x_{0})}\right)^{N-1}.
$$
If $a\leq x_{0}\leq x_{1}$: 
$$
\left(\frac{s_{K/(N-1)}(D-x_{1})}{s_{K/(N-1)}(D-x_{0})}\right)^{N-1}
\leq
\frac{h_{a}(x_{1})}{h_{a}(x_{0})}=
 \left(\frac{s_{K/(N-1)}(x_{1})}{s_{K/(N-1)}(x_{0})}\right)^{N-1}.
$$
If $ x_{0}\leq a\leq x_{1}$: 
$$
\frac{h_{a}(x_{1})}{h_{a}(x_{0})} =  
\left(\frac{s_{K/(N-1)}(x_{1})}{s_{K/(N-1)}(a)}\right)^{N-1} . 
\left(\frac{s_{K/(N-1)}(D- a)}{s_{K/(N-1)}(D-x_{0})}\right)^{N-1}; 
$$
 using again  the fact that \eqref{E:decreasing} is decreasing, we get the claim.
\end{proof}

\begin{lemma}
\label{lem:nocd}
For every choice of $K$, $N$ and $D$, except the case in which $K>0$ and $D= \pi \sqrt{(N-1)/K}$, the density $h_a$ defined in \eqref{equ:ha}  does not verify $\CD(K,N)$.
\end{lemma}
\begin{proof}
Recall that a non-negative Borel function $h$ defined on an interval $I\subset \mathbb{R}$ is called a $\CD(K,N)$ density  if for every $t\in[0,1]$  and for all $x_0,x_1\in{I}$ such that $x_0<x_1$, it holds:
\begin{equation}\label{equ:defcd}
h((1-t)x_0+t x_1)^{\frac{1}{N-1}} \geq \sigma^{(1-t)}_{K,N-1}(x_1-x_0)h(x_0)^{\frac{1}{N-1}}+ \sigma^{(t)}_{K,N-1}(x_1-x_0)h(x_1)^{\frac{1}{N-1}}.
\end{equation}
In order to prove our claim we will discuss several cases.\\
If $K=0$, the inequality  \eqref{equ:defcd} simply reduces to the   concavity of $h^{\frac{1}{N-1}}$.  We will prove now that \eqref{equ:defcd} fails for the density $h_a(\cdot)$ exactly for convex combinations that give out the point $a$. Pick $x_0<a<x_1$ and let $t\in{(0,1)}$ be such that $a=(1-t)x_0+t x_1$. It follows that
\begin{align*}
(1-t)h_a(x_0)^{\frac{1}{N-1}} +t h_a(x_1)^{\frac{1}{N-1}}&= f_{0,N,D}(a)^{\frac{1}{N-1}}\biggl[ (1-t)\biggl(\frac{D-x_0}{D-a}\biggr)+t\biggl(\frac{x_1}{a}\biggr) \biggr]\\
&>  f_{0,N,D}(a)^{\frac{1}{N-1}}=h_a(a)^{\frac{1}{N-1}},
\end{align*}
hence\eqref{equ:defcd} is not satisfied. \\ 
If $K\neq0$, we argue as follows. Since  $a=(1-t)x_0+t x_1$, it should be $t=\frac{a-x_0}{x_1-x_0}$ and $1-t=\frac{x_1-a}{x_1-x_0}$. Hence, we can rewrite the second member of the inequality \eqref{equ:defcd} in this form
\begin{equation}
 f_{K,N,D}(a)^{\frac{1}{N-1}}\biggl[\frac{s_{K/(N-1)}( x_{1}-a  )}{s_{K/(N-1)}(x_1- x_{0})}\cdot\frac{s_{K/(N-1)}(D-x_0)}{s_{K/(N-1)}(D-a)}  + \frac{s_{K/(N-1)}( a-x_{0}  )}{s_{K/(N-1)}(x_1- x_{0})}\cdot\frac{s_{K/(N-1)}(x_1)}{s_{K/(N-1)}(a)}      \biggr];
\end{equation}
using now that \eqref{E:decreasing} is a strictly  decreasing function, we get that the quantity above is strictly greater than
\begin{equation}
\label{E:boh}
 f_{K,N,D}(a)^{\frac{1}{N-1}}\biggl[\frac{s_{K/(N-1)}( x_{1}-a  )}{s_{K/(N-1)}(x_1- x_{0})}\cdot\frac{s_{K/(N-1)}(x_0)}{s_{K/(N-1)}(a)}  + \frac{s_{K/(N-1)}( a-x_{0}  )}{s_{K/(N-1)}(x_1- x_{0})}\cdot\frac{s_{K/(N-1)}(x_1)}{s_{K/(N-1)}(a)}      \biggr].
\end{equation}
If $K<0$,  assuming without loss of generality that $K=-(N-1)$, we get that \eqref{E:boh} can be rewritten in the following way
\[
 f_{-(N-1),N,D}(a)^{\frac{1}{N-1}}\biggl[\frac{\sinh( x_{1}-a  )\sinh(x_0)+\sinh(a-x_0)\sinh(x_1)}{\sinh(a)\sinh(x_1-x_0)} \biggr]=f_{-(N-1),N,D}(a)^{\frac{1}{N-1}},
\]
by straightforward computations.
Arguing in the same way in the case $K>0$ ( assuming as usual that $K=N-1$), we get that \eqref{E:boh} can be rewritten in this form
\[
 f_{N-1,N,D}(a)^{\frac{1}{N-1}}\biggl[\frac{\sin( x_{1}-a  )\sin(x_0)+\sin(a-x_0)\sin(x_1)}{\sin(a)\sin(x_1-x_0)} \biggr]=f_{N-1,N,D}(a)^{\frac{1}{N-1}}.
\]
Hence the claim follows also in this case.
\end{proof}

\bigskip
%%%%%%%%%%%%%%%%%%%%%%%%%%%%%%%%%%%%%%%%%%%%%%%%%%%%%%%%%%%%%%%
%%%%%%%%%%%%%%%%%%%%%%%%%%%%%%%%%%%%%%%%%%%%%%%%%%%%%%%%%%%%%%%
%%%%%%%%%%%%%%%%%%%%%%%%%%%%%%%%%%%%%%%%%%%%%%%%%%%%%%%%%%%%%%%

\subsection{One dimensional isoperimetric inequality}

To properly formulate the one-dimensional minimization problem,
let us consider the following set of probabilities
$$
\widetilde{\F}_{K,N,D}=\{ \mu\in{\P(\R)}:\mu=h_{\mu}\L^1,  h_{\mu}:[0,D]\to \mathbb{R},\ \ \MCP(K,N) \text{ density}\},
$$
and consider the following ``restricted'' minimization:
 for each $v \in (0,1)$
$$
\widetilde{\I}_{K,N,D} (v):=\inf \{ \mu^{+}(A): A\subset [0,D],\mu(A)=v, \mu\in{\widetilde{\F}_{K,N,D}} \}.
$$
The term ``restricted'' is motivated by the choice of fixing the 
domain of the $\MCP(K,N)$ densities.
For the ``unrestricted'' one-dimensional minimization we will adopt the classical notation 
\begin{equation}\label{E:normalminimal}
\I_{K,N,D} (v):=\inf \{ \mu^{+}(A): A\subset [0,D],\mu(A)=v, \mu\in \F_{K,N,D} \},
\end{equation}
where $\F_{K,N,D} = \cup_{D'\leq D} \widetilde{\F}_{K,N,D'}$.

\smallskip

The final claim will be to prove that each $h_{a}$ is a 
minimum of the isoperimetric problem for the volume equal to 
$\int_{(0,a)} h_{a}(x) \,dx$. We will therefore show that each volume $v \in (0,1)$ is reached in this manner.

\begin{lemma}\label{L:volume}
The map 
$$
(0,D) \ni a \longmapsto v(a) : = \int_{(0,a)} h_{a}(x)\, dx\,\, \in (0,1),
$$
is invertible.
\end{lemma}

\begin{proof}
It will be convenient to rewrite the function in the following way
\begin{equation}\label{E:anotherway}
v(a) = \frac{f_{K,N,D}(a)}{s_{K/(N-1)}^{N-1}(D-a)} \int_{(0,a)}
 s^{N-1}_{K/(N-1)}(D- x) \, dx
\end{equation}
implying differentiability. Given the strict monotonicity of the integral with respect to the variable $a$, it is sufficient to prove that also the other factor is an increasing function. Since 
$$
 \left(\frac{{s_{K/(N-1)}^{N-1}(D-a)}}{f_{K,N,D}(a)}\right)' = 
\left[ \left(\frac{s_{K/(N-1)}(D-a) }{s_{K/(N-1)}(a)} \right)^{N-1}\right]' 
\int_{(a,D)}
 s^{N-1}_{K/(N-1)}(x) \, dx,
$$
it follows that the previous derivative has the same  sign of the derivative of \eqref{E:decreasing}, thus it is  non positive and the claim follows.
\end{proof}
\noindent
Hence for each $K,N,D$ it is possible to define the inverse map of $v(a)$ from Lemma \ref{L:volume}:
\begin{equation*}
(0,1) \ni v \longmapsto a_{K,N,D}(v) \in (0,D),
\end{equation*}
with $a_{K,N,D}(v)$ the unique element such that
\begin{equation}
\label{equ:av}
 \,\int_{(0,a_{K,N,D}(v))} h_{a_{K,N,D}(v)}(x) \,dx = v. 
\end{equation}
%%%%%%%%%%%%%%%%%%%%%%%%%%%%%%%%%%%%%%%%%%%%%%%%%%%%%%%%%%%%%%%%%%%%%%%%%%%%%%%%%
%%%%%%%%%%%%%%%%%%%%%%%%%%%%%%%%%%%%%%%%%%%%%%%%%%%%%%%%%%%%%%%%%%%%%%%%%%%%%%%%%

\noindent
For ease of notation we will prefer in few places the shorter notation $a_{v}$ to denote $a_{K,N,D}(v)$.

\begin{remark}\label{R:symmetriv}
The function $v\mapsto a_{v}$ enjoys a simple symmetric property:
by definition we have that
\begin{align*}
1-v&= \frac{f_{K,N,D}(a_v)}{s_{K/(N-1)}^{N-1}(a_v)} \int_{(a_v,D)}
 s^{N-1}_{K/(N-1)}( x) \, dx\\
&=\frac{f_{K,N,D}(D-a_v)}{s_{K/(N-1)}^{N-1}(a_v)} \int_{(0,D-a_v)}
 s^{N-1}_{K/(N-1)}(D- x) \, dx\\
&=v(D-a_v),
\end{align*}
where the last identity follows from \eqref{E:anotherway}.
Since there exists a unique value $a_{1-v}\in{(0,D)}$ such that $v(a_{1-v})=1-v$, it turns out  that $a_{1-v}=D-a_v$. 
\end{remark}

\noindent
The first main result of this note is the following explicit 
formula for $\widetilde{\I}_{K,N,D}$.

\begin{theorem}\label{T:comparison}
For each volume $v\in{(0,1)}$, it holds 
$$
\widetilde{\I}_{K,N,D} (v) = f_{K,N,D} (a_{K,N,D}(v)).
$$ 
In particular, since $f_{K,N,D} (a_{K,N,D}(v)) =
h_{a_{K,N,D}(v)}(a_{K,N,D}(v))$, the lower bound is attained.
\end{theorem}

\noindent
For the proof of Theorem \ref{T:comparison} will be useful to consider the function 
$A_{K,N,D}:[0,D)\to [0,\infty)$   defined as follows:
\begin{equation}\label{equ:A}
A_{K,N,D}(a):=\frac{v(a)}{f_{K,N,D}(a)}=\int_{(0,a)}
\left(\frac{  s_{K/(N-1)}(D- x)}{s_{K/(N-1)}(D-a)} \right)^{N-1} \, dx.
\end{equation}
We will use that $[0,D) \ni a \mapsto A_{K,N,D}(a)$ is increasing; we postpone the proof of this fact at the end of the section.
From the symmetric property of $a_{v}$ observed few lines above, 
we obtain the analogous one for $A_{K,N,D}$:
\begin{equation}\label{E:symmA}
\frac{1-v}{A_{K,N,D}(D-a_v)}=\frac{v(D-a_v)f_{K,N,D}(D-a_v)}{v(D-a_v)}=f_{K,N,D}(a_v).
\end{equation}

\begin{proof}[Proof of Theorem \ref{T:comparison}]

Fix $K,N,D \in \R$ with $N>1$ and any $v\in (0,1)$. 
Consider $h_{a_{v}}$ and $h_{a_{1-v}}$ and notice that 
$$
\int_{(0,a_{v})} h_{a_{v}}(x) \,dx = 
\int_{(a_{1-v},D)} h_{a_{1-v}}(x) \,dx = v
$$
and  
$$
h_{a_{v}}(a_{v}) = f_{K,N,D} (a_{v}) = 
f_{K,N,D} (a_{1- v}) = h_{a_{1-v}} (a_{1-v}),
$$
where the second equality follows from $a_{1-v} = D - a_{v}$ and 
the symmetric property of $f_{K,N,D}$.
Hence  it is enough to show that for any $\MCP(K,N)$ density $h : [0,D] \to [0,\infty)$, the following inequality is valid
$$
\I_{h}(v) \geq f_{K,N,D} (a_{K,N,D}(v)).
$$
In the one-dimensional setting, taking the lowest possible Minkowski content or the lowest possible perimeter with respect to $h$ makes no difference 
(see \cite[Corollary 3.2]{CM16b}).
Hence fix any $h$ as above and a set $E$ of finite perimeter 
with respect to $h \L^{1}$. It follows that, up to a Lebesgue negligible set, $E = \cup_{i\in{\mathscr{I}}} [a_i,b_i]\subseteq [0,D]$, where $\mathscr{I}\subseteq\mathbb{N}$ is a set of indices, 
so that (see \cite[Proposition 3.1]{CM16b})
$$
\PP_{h}(E) = \sum_{i} h(a_{i}) + h(b_{i}),
$$
where $\PP_{h}$ denotes the perimeter with respect to $h$.
First notice that if any $a_{i}, b_{i}$ is in the interval 
having as boundary points $a_v$ and $D-a_v$, the claim is proved
$$
h(x)\geq f_{K,N,D}
(x) \geq \inf_{y\in [a_v,D-a_v]}f_{K,N,D}(y)=f_{K,N,D}(a_v);
$$
the same chain of inequalities is valid if $2a_{v}\geq D$.
So for each $i\in \mathscr{I}$, points $a_{i},b_{i} \notin (a_{v},D-a_{v})$ if $a_{v} \leq D/2$, or
$a_{i},b_{i} \notin (D-a_{v},a_{v})$ if $a_{v} \geq D/2$.

\medskip
It is convenient to assume with no loss in generality that $a_{v} \leq D-a_{v}$ and consider the following subsets of indices
$$
\mathscr{I}_{1} : = \{ i \in \mathscr{I}\colon a_{i} \geq D-a_{v} \}, \qquad
\mathscr{I}_{2}: = \{ i \in \mathscr{I}\colon b_{i} \leq a_{v} \};
$$
notice that $\mathscr{I}_{1} \cap \mathscr{I}_{2} = \emptyset$.

\noindent
{\bf Case 1.} $\mathscr{I}=\mathscr{I}_{1}$. \\
Then
\begin{align*}
v=&~\sum_{i\in{\mathscr{I}}}\int_{a_i}^{b_i} h(y)\,dy\leq\sum_{i\in{\mathscr{I}}}h(a_i) \int_{a_i}^D 
\left(\frac {s_{K/(N-1)}(y)}{s_{K/(N-1)}(a_i)}\right)^{N-1} dy  \\
= &~\sum_{i\in{\mathscr{I}}} h(a_i)A(D-a_i)\leq A(a_v)\sum_{i\in{\mathscr{I}}}h(a_i).
\end{align*}
\noindent
Hence, we get
\[
\sum_{i\in{\mathscr{I}}}(h(a_i)+h(b_i))\geq \sum_{i\in{\mathscr{I}}}h(a_i) \geq \frac{v}{A(a_v)}=f_{K,N,D}(a_v).
\]

\noindent
{\bf Case 2.} $\mathscr{I}=\mathscr{I}_{2}$. \\
It holds true
\begin{align*}
v=\sum_{i\in{\mathscr{I}}}\int_{a_i}^{b_i} h(y)dy&\leq \sum_{i\in{\mathscr{I}}}h(b_i)\int_{a_i}^{b_i}\left( \frac{ s_{K/(N-1)}(D-y)}{s_{K/(N-1)}(D-b_i)}\right)^{N-1}\,dy \\
 &\leq\sum_{i\in{\mathscr{I}}}h(b_i) A(b_i)
 \\
&\leq A(a_v)\sum_{i\in{\mathscr{I}}} h(b_i),
\end{align*}
\noindent
for the increasing monotonicity of the function $A(\cdot)$.
\medskip

\noindent
{\bf Case 3.} $\mathscr{I} \neq \mathscr{I}_{1} 
\cup \mathscr{I}_{2}$. \\
There exists $i\in \mathscr{I}$ such that 
$a_{i} \leq a_{v}, D-a_{v} \leq b_{i}$.   
Then
\begin{align*}
1-v &\leq \int_0^{a_i}h(y)\,dy+\int_{b_i}^Dh(y)\,dy \\
&\leq h(a_i)\int_{0}^{a_i}
\left(\frac{ s_{K/(N-1)}(D-y)}{s_{K/(N-1)}(D-a_i)}\right)^{N-1}\,dy +h(b_i)\int_{b_i}^D 
\left(\frac {s_{K/(N-1)}(y)}{s_{K/(N-1)}(b_i)}\right)^{N-1} \, dy \\
&= h(a_i)A(a_{i})+h(b_i) A(D-b_i)\\
&\leq A(D-a_v)[h(a_i)+h(b_i)],
\end{align*}
proving the claim. 

\medskip
\noindent
{\bf Case 4.} $\mathscr{I} = \mathscr{I}_{1} \cup\mathscr{I}_{2}$.  \\
We use the estimates of {\bf Case 2.} for $\mathscr{I}_{1}$  and the ones in {\bf Step 1.} for $\mathscr{I}_{2}$, so:
\[
v=\sum_{i\in{\mathscr{I}}}\int_{a_i}^{b_i} h(y)\,dy= \sum_{i\in{\mathscr{I}_{1}}} \int_{a_i}^{b_i} h(y)\,dy+\sum_{j\in{\mathscr{I}_{2}}} \int_{a_j}^{b_j} h(y)\,dy\leq A(a_{v}) (\sum_{i\in{\mathscr{I}_{1}}}h(a_i)+\sum_{j\in{\mathscr{I}_{2}}}h(b_j)).
\]
Hence, the claim is proved also in this class.
\end{proof}

\begin{lemma}\label{L:Aincreasing}
The function $A_{K,N,D}(\cdot)$ is strictly increasing on $[0,D)$. 
\end{lemma}

\begin{proof}
If we are in the case $K=0$, we get that 
$$
A_{0,N,D}(a)=\int_{(0,a)}
\left(\frac{ D- x}{D-a}\right)^{N-1}dx
$$
and so $A_{0,N,D}(\cdot)$ is trivially increasing. If $K<0$, without loss of generality we can assume $K=-(N-1)$. In this case we have 
$$
A_{-(N-1),N,D}(a)=\int_{(0,a)}
\left(\frac{ \sinh(D- x)dx}{\sinh(D-a)}\right)^{N-1} \, dx
$$
and so again we get the claim  by the   monotonicity of the hyperbolic sine.
If $K>0$, we can  directly deal with the case $D<\pi\sqrt{(N-1)/K}$. Assuming  $K=N-1$, we can rewrite~\eqref{equ:A} in the following way:
$$
A_{N-1,N,D}(a)=\int_{(0,a)}
\left(\frac{ \sin(D- x)}{\sin(D-a)}\right)^{N-1} \,dx.
$$
For sure this function is increasing for $a\in{[D-\pi/2,D)}$ by the  monotonicity of  $\sin(D-\cdot)$; so, if $D\leq \pi/2$, we are done.  If this is not the case, i.e.  $D>\pi/2$,  we have to prove that the same result holds in $[0,D-\pi/2)$.  Computing the first derivative we obtain that
\begin{align}
A'_{N-1,N,D}(a)&=1+(N-1)\frac{\cos(D-a) }{\sin^{N}(D-a)}\int_{(0,a)}
 \sin^{N-1}(D- x) \, dx\notag \\
&=1+\frac{N-1}{\tan(D-a)}A_{N-1,N.D}(a);
\end{align}
so $A(\cdot)$ is solution of a differential equation. In order to prove that $A(\cdot)$ is an increasing function, we will check that its first derivative is positive, i.e.
\[
A_{N-1,N,D}(a) \leq -\frac{\tan(D-a)}{N-1}:=g(a), \,\,\forall a\in{[0,D-\pi/2)}.
\]

\noindent 
For $a=0$ we have  $A_{N-1,N,D}(a)=0$ and $g(a)=-\frac{\tan D}{N-1}>0$, hence the inequality at the initial point holds true. In order to prove that it holds for every $a\in{[0,D-\pi/2)}$, we will check that g verifies the following differential inequality:
\[
g'(a)> 1+\frac{N-1}{\tan(D-a)}\cdot g(a).
\]
Since the choice of $g$  makes the second member identically equals to zero, it is sufficient to prove that $g'(a)>0$ for every $a\in{[0,D-\pi/2)}$. This trivially holds true since
\[
g'(a)=\frac{1}{(N-1)\cos^2(D-a)}>0.
\]
Hence, the claim follows also in this case.
\end{proof}

We now analyse the dependence of $\widetilde{\I}_{K,N,D} (v)$ on the diameter. 

\begin{lemma}
\label{lem:decreasingD}
Fix $N, D>0$ and $v\in{(0,1)}$. 
\begin{itemize}
\item[-] if $K \leq 0$,  the map $D\mapsto \widetilde{\I}_{K,N,D} (v)$  is strictly decreasing;
\item[-] if $K > 0$,  the map $D \mapsto D\, \widetilde{\I}_{K,N,D} (v)$  is non-decreasing;
\end{itemize}
\end{lemma}
\begin{proof}
Given any $\MCP(K,N)$ density $h$ with domain $[0,D]$, and any other $D'$ defining $g(x): =\frac{D}{D'}h(\frac{Dx}{D'})$, for each $x \in [0,D']$, 
one easily gets that $g$ is an $\MCP(K',N)$ with domain $[0,D']$ and 
$K'=K(D/D')^2$.
Moreover for any $A\subset [0,D]$, 
$$
\PP_{g}\left(A\frac{D'}{D}\right) = \frac{D}{D'} \PP_{h}(A),
$$
where $\PP_{g}$ is the perimeter with respect to $g$ and 
$\PP_{h}$ the one with respect to $h$.
Assume $h$ is the optimal density and $A$ the optimal set, one gets 
$$
\widetilde{\I}_{K',N,D'} \leq \frac{D}{D'}\widetilde{\I}_{K,N,D}.
$$
Hence if $K \leq 0$ and $D'\geq D$:
$
\widetilde{\I}_{K,N,D} \geq \frac{D'}{D}  \widetilde{\I}_{K,N,D'}\geq 
\widetilde{\I}_{K,N,D'};
$
if $K > 0$ and $D\geq D'$:
$
D\, \widetilde{\I}_{K,N,D} \geq D' \, \widetilde{\I}_{K,N,D'}.
$
The claim follows.
\end{proof}

We then obtain straightforwardly the next fact.

\begin{corollary}
The one-dimensional isoperimetric profile function \eqref{E:normalminimal}
has the following representation:
\begin{equation}
\label{equ:sharp1d}
\I_{K,N,D}(v)=
\begin{cases}
f_{K,N,D} (a_{K,N,D}(v)) \qquad \qquad \quad \textit{if}\,\,\, K\leq 0,\\
\inf_{D'\leq D} f_{K,N,D'} (a_{K,N,D'}(v))\,\,\quad \textit{if}\,\,\, K>0.
\end{cases}
\end{equation}
\end{corollary}

In the case $K>0$ we expect the map $D \mapsto f_{K,N,D'} (a_{K,N,D'}(v))$ to be strictly convex as some explicit calculations for particular choices of $v$ 
would suggest. However at the moment we cannot conclude 
the existence of a unique minimizer $\bar D = \bar D(K,N,D,v) <D$ representing $\I_{K,N,D}(v)$ in the case $K>0$. This in turn affects  rigidity of the equality case of the isoperimetric inequality in the regime $K>0$.

\bigskip
%%%%%%%%%%%%%%%%%%%%%%%%%%%%%%%%%%%%%%%%%%%%%%%%%%%%%%%%
%%%%%%%%%%%%%%%%%%%%%%%%%%%%%%%%%%%%%%%%%%%%%%%%%%%%%%%%
%%%%%%%%%%%%%%%%%%%%%%%%%%%%%%%%%%%%%%%%%%%%%%%%%%%%%%%%

\subsection{One-dimensional rigidity}

Building on Corollary \ref{C:lowerboundrigid}, we prove that 
the one-dimensional isoperimetric inequality obtained in
Theorem \ref{T:comparison} is rigid.

\begin{theorem}\label{T:rigidityonedim}
Let $h : [0,D] \to \R$ be a $\MCP(K,N)$ density which integrates to $1$. 
Assume there exists $v \in (0,1)$ such that 
$\mathcal{I}_{h}(v) =  \widetilde{\I}_{K,N,D}(v)$.
Then either $h = h_{a_{v}}$ or $h = h_{a_{1-v}}$.
\end{theorem}

\begin{proof}
Assume the existence of a sequence of sets $E_{i} \subset [0,D]$ so that 
$$
\int_{E_{i}}h(x)\,dx = v, \qquad \lim_{i \to \infty}(h\L^{1}\llcorner_{[0,D]})^{+}(E_{i})= \widetilde{\I}_{K,N,D}(v).
$$
Then one can find a sequence of sets having perimeter with respect to $h$ converging to $\widetilde{\I}_{K,N,D}(v)$ still with volume $v$. 
By lower-semicontinuity we deduce the existence of a set $\cup_{i\in{\mathscr{I}}} [a_i,b_i]$ of volume $v$ such that 
$$
\sum_{i}h(a_{i}) + h(b_{i})= f_{K,N,D} (a_{K,N,D}(v)).
$$
We then proceed as in the proof of Theorem \ref{T:comparison}. \\
In the
{\bf Case 1.}, $\mathscr{I}=\mathscr{I}_{1}$, the first chain of inequalities yields 
that $\cup_{i\in{\mathscr{I}}} [a_i,b_i] = [a_{1},D]$ and strict monotonicity of $A_{K,N,D}$ implies that $D- a_{1} = a_{v}$. 
The second chain of inequalities then implies 
$$
h(D-a_{v}) = f_{K,N,D} (a_{K,N,D}(v)) = f_{K,N,D} (D - a_{K,N,D}(v)). 
$$
Corollary \ref{C:lowerboundrigid} yields $h = h_{D-a_v}$ and the  set $\cup_{i\in{\mathscr{I}}} [a_i,b_i] = [D-a_{v}, D]$.   
Equality in {\bf Case 2.}, $\mathscr{I}=\mathscr{I}_{2}$, implies, repeating the same argument, that 
$h = h_{a_v}$ and the  set $\cup_{i\in{\mathscr{I}}} [a_i,b_i] = [0,a_{v}]$.   
Equality in {\bf Case 3.} cannot be achieved: the chain of inequality implies that 
$\cup_{i\in{\mathscr{I}}} [a_i,b_i] = [a_{1},b_{1}]$ and $a_{1} = a_{v}$ and $b_{1} = D-a_{v}$; coupled with the chain of inequality implies 
$$
f_{K,N,D}(a_{v}) = h (a_{v}) + h(D-a_{v}) \geq 2  f_{K,N,D}(a_{v}),
$$
giving a contradiction. The same argument implies that also equality in {\bf Case 4.} cannot be achieved.
\end{proof}

Exploiting Lemma \ref{lem:decreasingD},
in the case $K\leq 0$ one can obtain the following stronger rigidity
\begin{corollary}\label{C:realrigid}
Let $h : [0,D'] \to \R$ be a $\MCP(K,N)$ density which integrates to $1$ with $K\leq 0$. 
Assume there exists $v \in (0,1)$ such that 
$\mathcal{I}_{h}(v) =  \I_{K,N,D}(v)$ with $D' \leq D$.
Then $D = D'$ and either $h = h_{a_{v}}$ or $h = h_{a_{1-v}}$.
\end{corollary}
\begin{proof}
Lemma \ref{lem:decreasingD} forces $D' = D$ and then Theorem 
\ref{T:rigidityonedim} applies.
\end{proof}

To conclude we present another application of one-dimensional rigidity. 
Since $\CD(K,N)\subset\MCP(K,N)$, we already know that  $\widetilde{\I}_{K,N,D} (v)\leq \widetilde{\I}^{\CD}_{K,N,D} (v)$.
We can now prove that the inequality is always strict, made exception of a single case.

\begin{corollary}\label{C:comparison}
For every choice of $K$, $N$ and $D$, except the case in which $K>0$ and $D= \pi \sqrt{(N-1)/K}$, it holds  
$$
\widetilde{\I}_{K,N,D} (v) < \I^{\CD}_{K,N,D} (v).
$$
In particular, $\I_{K,N,D} (v) < \I^{\CD}_{K,N,D} (v)$.
\end{corollary}
\begin{proof}
Suppose by contradiction the existence of $K,N,D, v$ such that 
$\widetilde{\I}_{K,N,D} (v) = \I^{\CD}_{K,N,D} (v)$.
As proved in \cite{Mil}(see Corollary 1.4)
$$
\I^{\CD}_{K,N,D} (v) = \widetilde{\I}^{\CD}_{K,N,D} (v),
$$
and there exists (see \cite[Corollary A.3]{Mil}) a $\CD(K,N)$-density, and therefore an $\MCP(K,N)$-density $g$ defined on $[0,D]$ and integrating to $1$ such that  $\I_{([0,D], g)}(v)= {\I}^{\CD}_{K,N,D} (v)$. As observed in the   Theorem \ref{T:rigidityonedim}, this would force the density $g$  to be  exactly $h_{a_v}$ or  $h_{a_{1-v}}$ contradicting Lemma \ref{lem:nocd}.
The final claim simply follows observing that
$\inf_{D'\leq D} \widetilde{\I}_{K,N,D'} (v) \leq \I_{K,N, D} (v)$.
\end{proof}

\bigskip

%%%%%%%%%%%%%%%%%%%%%%%%%%%%%%%%%%%%%%%%%%%%%%%%%%%%%%%%%%%%
%%%%%%%%%%%%%%%%%%%%%%%%%%%%%%%%%%%%%%%%%%%%%%%%%%%%%%%%%%%%
%%%%%%%%%%%%%%%%%%%%%%%%%%%%%%%%%%%%%%%%%%%%%%%%%%%%%%%%%%%%
%%%%%%%%%%%%%%%%%%%%%%%%%%%%%%%%%%%%%%%%%%%%%%%%%%%%%%%%%%%%
%%%%%%%%%%%%%%%%%%%%%%%%%%%%%%%%%%%%%%%%%%%%%%%%%%%%%%%%%%%%
\section{Isoperimetric inequality}\label{S:final}

We now deduce Theorem \ref{T:ISOMCP} from the one-dimensional results 
of Theorem \ref{T:comparison} and Lemma \ref{lem:decreasingD} via localization techniques; we now briefly recall few facts on localization. 

The localization paradigm, developed by Payne--Weinberger \cite{PayneWeinberger}, Gromov--Milman \cite{Gromov-Milman} and Kannan--Lov\'asz--Simonovits \cite{KLS}, permits to reduce various analytic and geometric inequalities to appropriate one-dimensional counterparts. The original approach by these authors was based on a bisection method, and thus inherently confined to $\Real^n$. In 2015 \cite{klartag}, Klartag extended the localization paradigm to the weighted Riemannian setting, by disintegrating the reference measure $\mm$ on $L^1$-Optimal-Transport geodesics associated to the inequality under study, and proving that the resulting conditional one-dimensional measures inherit the Curvature-Dimension properties of the underlying manifold.

The first author and Mondino in \cite{CM1} extended the localization paradigm to the framework of essentially non-branching geodesic m.m.s.'s $(X,\sfd,\mm)$ verifying $\CD_{loc}(K,N)$, $N \in (1,\infty)$:  the Curvature-Dimension information encoded in the $W_2$-geodesics is transferred to the individual rays along which a given $W_1$-geodesic evolves; this has permitted to obtain several new results in the field 
\cite{CM2,CM17,CMM}.

Localization for $\MCP(K,N)$ was, partially and in a different form, already known 
in 2009, see \cite[Theorem 9.5]{biacava:streconv}, for non-branching m.m.s.. 
The case of essentially non-branching m.m.s.'s and an effective reformulation (after the work of Klartag \cite{klartag}) has been recently discussed in \cite[Section 3]{CM18} to which we refer for all the missing details (see in particular \cite[Theorem 3.5]{CM18}). 
Here we only report the next fact:

\smallskip
If $(X,\sfd,\mm)$ is an essentially non-branching m.m.s. with $\supp(\mm) = X$ and satisfying $\MCP(K,N)$, for some $K\in \R, N\in (1,\infty)$,  
then, for any $1$-Lipschitz function $u : X \to \R$, 
the non-branching transport set $\T_{u}^{b}$ associated with $u$ 
(roughly coinciding, up to a set of $\mm$-measure zero, with $\{|\nabla u| =1\}$) admits a disjoint family of unparametrized geodesics $\{ X_{\alpha} \}_{\alpha \in Q}$ such that $\mm(\T_{u}^{b} \setminus \cup_{\alpha} X_{\alpha})= 0$
and the corresponding disintegration of $\mm$ is as follows
\begin{equation}\label{E:disint}
\mm\llcorner_{\T_{u}^{b}} = \int_{Q} \mm_{\alpha} \, \qq(d\alpha), \qquad \qq(Q) = 1, \qquad \qq{\rm -a.e.}\ \ \mm_{\alpha}(X) = \mm_{\alpha}(X_{\alpha})=1.
\end{equation}
Moreover, $\qq$-a.e. $\mm_{\alpha}$ is a Radon measure  with $\mm_{\alpha}=h_{\alpha} \cH^{1}\llcorner_{X_{\alpha}} \ll \cH^{1}\llcorner_{X_{\alpha}}$
and $(X_{\alpha},\sfd,\mm_{\alpha})$ verifies $\MCP(K,N)$.

This permits  to obtain the next main result and to prove Theorem \ref{T:ISOMCP}; notice that the second part of 
Theorem \ref{T:ISOMCP} will then follow by Theorem \ref{T:comparison}.

\begin{theorem}
\label{teo:main}
Let $(X,\sfd,\mm)$ be an essentially non-branching metric measure space    with $\mm(X)=1$ and $\diam(X)\leq D$. If $(X,\sfd,\mm)$ satisfies $\MCP(K,N)$ for some $K\in{\R}, N\in{[1, \infty)}$, then  
\[
\I_{(X,\sfd,\mm)}(v)\geq \I_{K,N,D}(v),\,\,\,\,\forall v\in{[0,1]}
\]
where $\I_{K,N,D}$ is explicetely given in \eqref{equ:sharp1d}.
\end{theorem}

Even though the proof is a standard consequence of localization, we present it 
below for readers' convenience. 

\begin{proof}
Fix $v\in(0,1)$ and let $A\subset X$ be a Borel set with $\mm(A) = v$.
Define the $\mm$-measurable function $f:=\chi_A-v$ having zero integral with respect to $\mm$, and study the $L^{1}$-Optimal Transport problem from $\mu_{0} : = f^{+}\mm$ to $\mu_{1} : = f^{-}\mm$, where $f^{\pm} $ denotes the positive and the negative part of $f$ respectively. 
The associated Kantorovich potential $u$ has $|\nabla u| = 1$ $\mm$-a.e.
implying the existence of a family of unparametrized geodesics $\{ X_{\alpha} \}_{\alpha \in Q}$ (of length at most $D$) such that $\mm(X \setminus \cup_{\alpha} X_{\alpha})= 0$ and
$$
\mm= \int_{Q} \mm_\alpha \,  \q(d\alpha), \qquad \qq-a.e. \ \mm_{\alpha}(X) 
= \mm_{\alpha}(X_{\alpha}) =1;
$$
moreover $\mm_\alpha = h_\alpha \H^1\llcorner_{X_{\alpha}}$ and $h_{\alpha}$ is a 
$\MCP(K,N)$-density.
From the localization of the constraint, it follows that for $\q$-a.e. 
$\mm_{\alpha}(A ) = \mm (A) = v$.
Hence 
\begin{align*}
\mm^{+}(A)&=\liminf_{\varepsilon\to 0}
\frac{ \mm(A^{\varepsilon})-\mm(A)}{\varepsilon}\\
&\geq \liminf_{\varepsilon\to 0} \int_{\mathcal{Q}} 
\frac{ \mm_{\alpha}((A\cap X_{\alpha})^{\varepsilon})-\mm_{\alpha}(A)}{\varepsilon} \,\q(d\alpha),\\
&\geq \int_{\mathcal{Q}} \mm^{+}_{\alpha}(A\cap X_{\alpha})\, \q(dq)\\
&\geq \int_{Q} \I_{K,N,D}(v)\,\q(q)\\
&= \I_{K,N,D}(v).
\end{align*}
\end{proof}

In the case $K\leq 0$, one-dimensional rigidity (Theorem \ref{T:rigidityonedim}) implies the following measure rigidity.

\begin{theorem}\label{T:rigidityfull}
Let $(X,\sfd,\mm)$ be an essentially non-branching metric measure space  satisfying    $\MCP(K,N)$ for $K\leq0$, $N\in{[1,\infty)}$  with $m(X)=1$ and $\diam(X)\leq D$.  

If there exists $v\in{(0,1)}$ such that $\I_{(X,d,m)}(v)=\I_{K,N,D}(v)$,  then $\diam(X) = D$, there exist a measure space $(Q,\q)$  and a measurable isomorfism between $(0,D)\times Q$   and $X' \subset X$ with $\mm(X')=1$. 

Moreover, the measure $\mm$ admits the following representation 
$$
\mm = \int_{Q} h_{\alpha}\,\H^1\llcorner_{X_{\alpha}} \, \q(d\alpha),
$$
and $\qq$-a.e., $h_{\alpha} = h_{a_{K,N,D}(v)}$ or 
$h_{\alpha} = h_{a_{K,N,D}(1-v)}$.
\end{theorem}

\begin{proof}
We will prove that $X$ has diameter $D$. Arguing by contradiction, let us suppose that  there exists $\varepsilon>0$ such that $\diam(X)=D-\varepsilon$.
From \eqref{equ:sharp1d}, $K \leq 0$ and  Lemma \ref{lem:decreasingD}, for any $v\in{(0,1)}$ the function $\I_{K,N,D}(v)$ is strictly decreasing in $D$. Hence, there exists $\eta>0$ such that
\[
\I_{K,N,D'}(v) \geq \I_{K,N,D}(v)+\eta,\quad\forall D'\in{(0, D-\varepsilon]}.
\]
Let $A\subset X$ be such that $\mm(A)=v$ and $\mm^{+}(A)\leq \I_{K,N,D}(v)+\eta/2$. Arguing as in the proof of  Theorem \ref{teo:main}, we get that
\begin{align*}
\I_{K,N,D}(v)+\eta/2 \geq \mm^{+}(A) 
&\geq \int_{Q} \mm_{\alpha}^{+} ( A \cap X_{\alpha})\, \q(d\alpha) \\
&\geq \int_{Q} \widetilde{\I}_{K,N,|\supp\, h_\alpha|}(v)\,\q(d\alpha)\\
&\geq \I_{K,N,D}(v)+\eta
\end{align*}
where the last inequality is due to the fact that   $\supp (h_\alpha)$ is isometric to a geodesic $X_\alpha$ of $(X,\sfd)$ and hence $|\supp\,  h_\alpha|\leq D-\varepsilon$ and from $K\leq 0$ together with Lemma \ref{lem:decreasingD}. 
Thus the contradiction is obtained. \\ 
The same argument implies that $|\supp(h_\alpha)|=D$ for $\q$-a.e. $\alpha$ and 
$$
\I_{K,N,D}(v) = \widetilde{\I}_{K,N,|\supp\, h_q|}(v);
$$ 
the claim follows from the one-dimensional rigidity obtained in Corollary 
\ref{C:realrigid}.
\end{proof}


\bigskip
\begin{thebibliography}{10}


%% Use the widest label as parameter.

%% Reference items may be numbered or have labels of your choice.
%% The author's surname PRECEDES the initial of the first name
%% The issue number is only given when the issues are paginated separately.
%% In book titles, first letters are capitalized.
%% Only journal volume numbers are boldfaced.




\footnotesize
\baselineskip=12pt

\bibitem{AmbrosioICM} 
L. Ambrosio.
Calculus and curvature-dimension bounds in metric measure spaces. 
\emph{Proceedings of the ICM 2018}, preprint available at http://cvgmt.sns.it/paper/3779/.
   
%
%
%\bibitem{AGMR2012}
%{\sc Ambrosio L.,  Gigli N., Mondino A. and Rajala T.},
%  \emph{Riemannian Ricci curvature lower bounds in metric spaces with $\sigma$-finite measure},  Trans. Amer. Math. Soc., \textbf{367}, no. 7,  (2015), 4661--4701.
%   
%
%
\bibitem{AGS11a}
 L. Ambrosio, N. Gigli and G. Savar\'e.
Calculus and heat flow in metric measure spaces and applications to spaces 
with {R}icci bounds from  below.  
 {\em Invent. Math.}, \textbf{195}, no. 2,  (2014), 289--391.
  
  
\bibitem{AGS11b}
\leavevmode\vrule height 2pt depth -1.6pt width 23pt:
 {\em Metric measure
  spaces with {R}iemannian {R}icci curvature bounded from below},  Duke Math. Journ., \textbf{163}, no. 7,  (2014),  1405--1490.


\bibitem{AGS12}
\leavevmode\vrule height 2pt depth -1.6pt width 23pt:
 {\em Bakry-\'{E}mery
  curvature-dimension condition and {R}iemannian {R}icci curvature bounds},
  Ann. Probab., \textbf{43}, no. 1,  (2015), 339--404.
  


\bibitem{AMS2013} 
 L. Ambrosio, A. Mondino and G. Savar\'e.
Nonlinear diffusion equations and curvature conditions in
metric measure spaces. {\em Mem. Amer. Math. Soc.}, in press.
%
%
%\bibitem{BS10}  {\sc  Bacher K. and Sturm K.T.,}
%{\em Localization and tensorization properties of the curvature-dimension condition for metric measure spaces,}
% J. Funct. Anal.,   \textbf{259}, (2010),  28--56.
%

\bibitem{BakryStFlour}
D.~Bakry.
\newblock L'hypercontractivit\'e et son utilisation en th\'eorie des
  semigroupes.
\newblock In {\em Lectures on probability theory ({S}aint-{F}lour, 1992)},
  volume 1581 of {\em Lecture Notes in Math.}, pages 1--114. Springer, Berlin,
  1994.

\bibitem{BakryEmery}
D.~Bakry and M.~{\'E}mery.
\newblock Diffusions hypercontractives.
\newblock In {\em S\'eminaire de probabilit\'es, XIX, 1983/84}, volume 1123 of
  {\em Lecture Notes in Math.}, pages 177--206. Springer, Berlin, 1985.

\bibitem{Balogh}
Z. M. Balogh, A. Krist‡ly, and K. Sipos. 
Geometric inequalities on Heisenberg groups. 
{\em Calc. Var. Partial Differential Equations}, 57 (2018), no. 2, Art. 61, 41. 


\bibitem{BarilariRizzi}  D. Barilari and L. Rizzi.
 Sharp measure contraction property for generalized $H$-type Carnot groups.  
 {\em Commun. Contemp. Math.}. 20 (2018), no. 06, 1750081.
%

\bibitem{BarilariRizzi2}  D. Barilari and L. Rizzi.
Sub-Riemannian interpolation inequalities.  
Preprint arXiv:1705.05380.
%


%
\bibitem{biacava:streconv}
S.~Bianchini and F.~Cavalletti.
\newblock The {M}onge problem for distance cost in geodesic spaces.
\newblock {\em Comm. Math. Phys},  {318}:615 -- 673, 2013.

%
%\bibitem{BB}
%{\sc Bj{\"o}rn A. and Bj{\"o}rn} J., {\em Nonlinear potential theory on metric
%  spaces}, \textbf{17}, EMS Tracts in Mathematics, European Mathematical Society
%  (EMS), Z\"urich, (2011).
%
%
%\bibitem{Calabi} {\sc Calabi E.},
%{\em An extension of E. Hopf's maximum principle with an application to Riemannian geometry},
%Duke Math. Journ, \textbf{25}, (1),  (1958), 45--56.    
%
%
%
\bibitem{Capogna} L. Capogna, D. Danielli, S. D. Pauls and J. Tyson. 
An introduction to the Heisenberg group and the sub-Riemannian isoperimetric problem. {\em Progress in Mathematics}, 259. Birkh\"auser Verlag, Basel, 2007. xvi+223 pp.


\bibitem{BrueSemola} E. Bru\`e and D. Semola.
Constancy of the dimension for $\RCD(K,N)$ spaces via regularity of Lagrangian flows. 
{\em Comm. Pure Appl. Math.}, accepted for publication.

\bibitem{cava:MongeRCD}  F. Cavalletti.
\newblock  Monge problem in metric measure spaces with Riemannian curvature-dimension condition.
{\em Nonlinear Anal.}, \textbf{99}, (2014), 136--151.
%

\bibitem{cava:decomposition} F.~Cavalletti. %\leavevmode\vrule height 2pt depth -1.6pt width 23pt,
\newblock Decomposition of geodesics in the {W}asserstein space and the  globalization property.
\newblock {\em Geom. Funct. Anal.},  {24}:493 -- 551, 2014.


\bibitem{cava:overview} F.~Cavalletti. %\leavevmode\vrule height 2pt depth -1.6pt width 23pt,
\newblock An Overview of $L^{1}$ optimal transportation on metric measure spaces.
\newblock {\em Book Chapter}, to appear in ``Measure Theory in Non-Smooth Spaces",
\newblock edited by N. Gigli, De Gruyter Open. 

\bibitem{CMM}
F.~Cavalletti, F.~Maggi and A.~Mondino.
Quantitative isoperimetry \`a la Levy-Gromov.
{\em Comm. Pure Appl. Math},  to appear.


\bibitem{CMi16}
F. Cavalletti and E. Milman.
The Globalization Theorem for the Curvature-Dimension Condition.
Preprint arXiv:1612.07623.

%
\bibitem{CM1}
F.~Cavalletti and A.~Mondino.
Sharp and rigid isoperimetric inequalities in metric-measure spaces with lower Ricci curvature bounds.
{\em Invent. Math.},  \textbf{208}, 3,  (2017),  803--849.

\bibitem{CM2} 
\leavevmode\vrule height 2pt depth -1.6pt width 23pt,
\newblock Sharp geometric and functional inequalities in metric measure spaces with lower Ricci curvature bounds.
\newblock {\em Geom. Topol.}, to appear, arXiv:1502.06465.



\bibitem{CM16}
\leavevmode\vrule height 2pt depth -1.6pt width 23pt:
Optimal maps in essentially non-branching spaces,
{\em  Commun. Contemp. Math.}, \textbf{19}, (6), (2017),  1750007 (27 pages).
 
\bibitem{CM16b}
\leavevmode\vrule height 2pt depth -1.6pt width 23pt:
Isoperimetric inequalities for finite perimeter sets under lower Ricci curvature bounds.
{\em Rend. Lincei Mat. Appl.}, 29 (2018), 413Ð430 DOI 10.4171/RLM/814.

 
 
 \bibitem{CM17}
\leavevmode\vrule height 2pt depth -1.6pt width 23pt:
Almost euclidean Isoperimetric Inequalities in spaces satisfying local  Ricci curvature lower bounds.
{\em Int. Math. Res. Not.}, https://doi.org/10.1093/imrn/rny070.
% 

\bibitem{CM18}
\leavevmode\vrule height 2pt depth -1.6pt width 23pt:
New formulas for the Laplacian of distance functions and applications.
{\em Preprint},	arXiv:1803.09687.


% 

\bibitem{corderomccann:brescamp}
D.~Cordero-Erausquin, R.~J. McCann, and M.~Schmuckenshl\"ager.
\newblock A {R}iemannian interpolation inequality \`a la {B}orell, {B}rascamp and {L}ieb.
\newblock {\em Invent. Math.},  {146}:219--257, 2001.



%
\bibitem{EKS2013} M. Erbar, K. Kuwada and K.T. Sturm.
On the equivalence of the entropic curvature-dimension condition and {B}ochner's inequality on metric measure spaces. 
 {\em Invent. Math.}, \textbf{201}, no. 3,  (2015), 993--1071.
%  
%



\bibitem{Gigli12}
 N. Gigli.  
 On the differential structure of metric measure spaces and applications.
  {\em Mem. Amer. Math. Soc.}, \textbf{236}, no. 1113,  (2015).  
  
  \bibitem{GigliSplitting}
\leavevmode\vrule height 2pt depth -1.6pt width 23pt:
{\em The splitting theorem in non-smooth context}, preprint arXiv:1302.5555, (2013).

%
%\bibitem{GiMo} {\sc Gigli N. and Mondino A.}
%{\em A PDE approach to nonlinear potential theory in metric measure spaces}, 
%J. Math. Pures  Appl., \textbf{100}, no. 4, (2013),  505--534. 
%
%
%\bibitem{GMR2013} 
%{\sc Gigli N., Mondino A. and Rajala T.},
%{\em Euclidean spaces as weak tangents of infinitesimally Hilbertian metric measure spaces with Ricci curvature bounded below}, J. Reine Angew. Math., \textbf{705},  (2015), 233--244. 
%
%
%\bibitem{GMS2013}
%{\sc Gigli N.,  Mondino A. and Savar\'e G.},
%{\em Convergence of pointed non-compact metric measure spaces and stability of Ricci curvature bounds and heat flows},  Proc. London Math. Soc.,  \textbf{111}, no. 5, (2015), 1071--1129. 
%
%
\bibitem{GigliPasqualetto}
N. Gigli and E. Pasqualetto.
Behaviour of the reference measure on $\RCD$ spaces under charts.
{\em Comm. Anal. Geom.}, to appear, arXiv:1607.05188..


\bibitem{Gro}  M.~Gromov.
\newblock Metric structures for Riemannian and non Riemannian spaces,
{\em  Modern Birkh\"auser Classics}, (2007).

\bibitem{Gromov-Milman}
M.~Gromov and V.~D. Milman.
\newblock Generalization of the spherical isoperimetric inequality to uniformly
  convex {B}anach spaces.
\newblock {\em Compositio Math.}, 62(3):263--282, 1987.


\bibitem{KLS}
R.~Kannan, L.~Lov{\'a}sz, and M.~Simonovits.
\newblock Isoperimetric problems for convex bodies and a localization lemma.
\newblock {\em Discrete Comput. Geom.}, 13(3-4):541--559, 1995.


\bibitem{KellMondino} 
M. Kell and A. Mondino.
On the volume measure of non-smooth spaces with Ricci curvature bounded below. 
{\em Annali SNS-Classe di Scienze},  DOI Number: $10.2422/2036-2145.201608\_007$.
%
%\bibitem{Ket} {\sc Ketterer C.},
%{\em Cones over metric measure spaces and the maximal diameter theorem},
%J. Math. Pures Appl., \textbf{103}, no. 5,  (2015),  1228--1275.
%
%
%\bibitem{KitabeppuLakzian} {\sc Kitabeppu Y. and Lakzian S.},
%{\em Characterization of Low Dimensional $\RCD^{*}(K, N)$ Spaces}, Anal. Geom. Metr. Spaces, \textbf{4},  (2016), 187--215.
%
%
\bibitem{klartag}  B. Klartag.
Needle decomposition in Riemannian geometry.  
{\em Mem.  Amer. Math. Soc.},  \textbf{249}, no.1180,  (2017).
%
%
%
\bibitem{Juillet}  N. Juillet.
Geometric Inequalities and Generalized Ricci Bounds in the Heisenberg Group.
{\em Int. Math. Res. Not.},  \textbf{13},  (2009), 2347--2373.
%
%
\bibitem{Lott-Villani09}
J. Lott and C. Villani.
Ricci curvature for metric-measure spaces via optimal transport.
{\em Annals of Math.}, \textbf{169}, (2009), 903--991.
%
%

\bibitem{Mil} E.~Milman.
\newblock Sharp Isoperimetric Inequalities and Model Spaces for Curvature-Dimension-Diameter Condition.
\newblock{\em J. Europ. Math. Soc.},  \textbf{17}, (5),  (2015), 1041--1078.



\bibitem{MondinoNaber} A.Mondino and A.Naber.	
Structure Theory of Metric-Measure Spaces with Lower Ricci Curvature Bounds.
{\em J. Europ. Math. Soc.}, to appear,  arXiv:1405.2222..


\bibitem{Monti} R. Monti.
\newblock Isoperimetric problem and minimal surfaces in the Heisenberg group.
 In: Ambrosio L. (eds) Geometric Measure Theory and Real Analysis. 
{\em Publications of the Scuola Normale Superiore}, vol 17. Edizioni della Normale, Pisa



\bibitem{Ohta1} S.I. Ohta. 
On the measure contraction property of metric measure spaces.
 {\em Comment. Math. Helv.}, \textbf{82} (2007), 805--828.



%
\bibitem{OhtaJLMS} \leavevmode\vrule height 2pt depth -1.6pt width 23pt:
Products, cones, and suspensions of spaces with the measure contraction property.
{\em J. London Math. Soc.} (2), \textbf{76}, (2007),  225--236.

\bibitem{OttoVillaniHWI}
F.~Otto and C.~Villani.
\newblock Generalization of an inequality by {T}alagrand and links with the
  logarithmic {S}obolev inequality.
\newblock {\em J. Funct. Anal.}, 173(2):361--400, 2000.


\bibitem{Pansu} P. Pansu. 
Une in\'egalit\'e isop\'erim\'etrique sur le groupe de Heisenberg. 
{\em C. R. Acad. Sci. Paris S\'er. I Math.}. 295 (1982), 127--130.


\bibitem{PayneWeinberger}
L.~E. Payne and H.~F. Weinberger.
\newblock An optimal {P}oincar\'e inequality for convex domains.
\newblock {\em Arch. Rational Mech. Anal.}, 5:286--292, 1960.

\bibitem{R2016} T.~Rajala.
\newblock {Failure of the local-to-global property for $\CD(K,N)$ spaces},
\newblock{\em Ann. Sc. Norm. Super. Pisa Cl. Sci.},  {16}:45--68, 2016.


\bibitem{RS2014}  T. Rajala and  K.T. Sturm.
Non-branching geodesics and optimal maps in strong $\CD(K,\infty)$-spaces.
{\em Calc. Var. Partial Differential Equations}, \textbf{50},  (2014), 831--846.
 
 
\bibitem{VonRenesseSturm} M.-K.~von Renesse and K.-T.~Sturm.
\newblock {Transport inequalities, gradient estimates, entropy and Ricci curvature},
\newblock{\em Comm. Pure Appl. Math.}, 58:923--940, 2005.


 
 \bibitem{Rifford}  L. Rifford.
Ricci curvatures in Carnot groups. 
{\em  Math. Control Relat. Fields}, \textbf{3}, (4), (2013), 467--487.
%
\bibitem{Ritore1} M. Ritor\'e.
Examples of area-minimizing surfaces in the sub-Riemannian Heisenberg group H1 with low regularity. 
{\em Calc. Var. Partial Differential Equations}, 34 (2009), no. 2, 179--192.

\bibitem{Ritore2} M. Ritor\'e.
A proof by calibration of an isoperimetric inequality in the Heisenberg group Hn. 
{\em Calc. Var. Partial Differential Equations} 44 (2012), no. 1-2, 47--60.

\bibitem{Sturm06I}  K.T. Sturm.
On the geometry of
  metric measure spaces. {I}. {\em Acta Math.}, \textbf{196} (2006),  65--131.

\bibitem{Sturm06II}
\leavevmode\vrule height 2pt depth -1.6pt width 23pt:
On the geometry of
  metric measure spaces. {II}.  {\em Acta Math.}, \textbf{196} (2006), 133--177.

\bibitem{Villani09}
C. Villani. Optimal transport. Old and new.  {\em Grundlehren
  der Mathematischen Wissenschaften}, \textbf{338}, Springer-Verlag, Berlin, (2009).
%

\end{thebibliography}
\end{document}